\documentclass[11pt,reqno]{article}
\usepackage{amssymb,amsmath,amsfonts}
\usepackage{amsthm}
\usepackage{graphicx}
\usepackage{epstopdf}
\usepackage{bm}
\usepackage{paralist}
\usepackage{color}
\usepackage{fullpage}
\usepackage{algorithm}
\usepackage[noend]{algorithmic}

\newtheorem{claim}{Claim}[section]
\newtheorem{lemma}[claim]{Lemma}

\newtheorem{theorem}{Theorem}

\theoremstyle{definition}

\newtheorem{remark}[claim]{Remark}

\def\<{\langle}
\def\>{\rangle}

\def\projp{{\boldsymbol P}^{\perp}}
\def\bg{{\boldsymbol g}}
\def\bx{{\boldsymbol x}}
\def\bz{{\boldsymbol z}}

\def\bxz{{\boldsymbol x_0}}

\def\bM{{\boldsymbol M}}
\def\bX{{\boldsymbol X}}

\def\bA{{\boldsymbol A}}

\def\bB{{\boldsymbol B}}

\def\bW{{\boldsymbol W}}

\def\tz{\tilde{z}}

\def\bhu{{\hat{\boldsymbol u}}}

\def\tbz{\tilde{{\boldsymbol z}}}
\def\tbu{\tilde{{\boldsymbol u}}}

\def\bx{{\boldsymbol x}}

\def\bv{{\boldsymbol v}}
\def\bu{{\boldsymbol u}}

\def\by{{\boldsymbol y}}
\def\id{{\rm I}}
\def\sT{{\sf T}}

\def\bsigma{{\boldsymbol{\sigma}}}

\def\reals{{\mathbb{R}}}

\def\integers{{\mathbb{Z}}}

\def\normal{{\sf N}}
 
\def\PSD{{\sf PSD}}
\def\SDP{{\sf SDP}}
\def\OPT{{\sf OPT}}

\def\sS{{\sf S}}
\def\bDelta{{\boldsymbol \Delta}}
\def\bLambda{{\boldsymbol \Lambda}}
\def\bSigma{{\boldsymbol \Sigma}}

\def\rank{{\rm rank}}
\def\Tr{{\sf Tr}}

\title{A Grothendieck-type inequality for local maxima}

\author{
Andrea~Montanari\footnote{Department of Electrical
    Engineering and Department of Statistics, Stanford University}
}

\date{\today}                                           

\begin{document}

\maketitle

\begin{abstract}
A large number of problems in optimization, machine learning, signal processing can be effectively
addressed by suitable semidefinite programming (SDP) relaxations. Unfortunately, generic SDP solvers hardly scale 
beyond  instances with a few hundreds variables (in the underlying combinatorial problem). 
On the other hand, it has been observed empirically that
an effective strategy amounts to introducing a (non-convex) rank constraint, and solving the resulting smooth 
optimization problem by ascent methods. This  non-convex problem has --generically--
a large number of local maxima, and the reason for this success is therefore unclear.

This paper provides rigorous support for this approach. For the problem of maximizing a linear functional over
the elliptope,  we prove that \emph{all local maxima} are within 
a small gap from the SDP optimum. In several problems of interest, arbitrarily small relative error can be achieved by 
taking the rank constraint $k$ to be of order one, independently of the problem size.
\end{abstract}

\pagenumbering{arabic}

\section{Motivation and result}

Let $\PSD(n) \equiv\{\bX\in\reals^{n\times n}:\; \bX\succeq
0\}$  be the cone of $n\times n$ symmetric positive semidefinite
matrice. The convex set of positive-semidefinite matrices with diagonal
entries equal to one will be denoted by
\begin{align}
\PSD_1(n) \equiv\big\{\bX\in\reals^{n\times n}:\; \bX\succeq
0, \;X_{ii}=1\forall i\in [n]\big\}\, .
\end{align}
The set $\PSD_1(n)$  is also known as the \emph{elliptope}. Given a
symmetric matrix $\bA$, we define\footnote{Here and below
  $\<\bA,\bB\>=\Tr(\bA^{\sT}\bB)$ is the usual scalar product between matrices.}
\begin{align}
\SDP(\bA) \equiv \max\big\{ \<\bA,\bX\>\, :\;\;
\bX\in\PSD_1(n)\big\}\, .  \label{eq:SDP.DEF}
\end{align}
This semidefinite program (SDP) arises in a large number of applications
in particular as a relaxation of max-cut, and min-bisection of graphs.
Early references include \cite{goemans1995improved,nesterov1998semidefinite}.
Unfortunately, despite the many remarkable properties of this SDP, generic SDP solvers hardly scale beyond 
$n$ of the order of a few hundreds.

Recently, several authors \cite{burer2003nonlinear,OursReplicas} found that an effective strategy is to
constrain the rank of $\bX$ to satisfy 
$\rank(\bX)\le k$, explicitly solve the PSD constraint and use gradient ascent or coordinate ascent 
to maximize the resulting non-convex objective.

Explicitly, the rank-constrained problem and be written as a smooth optimization problem over the manifold
$\sS(n,k) \subset \reals^{n\times k}$, defined by
$\sS(n,k) \equiv\{\bsigma = (\bsigma_1,\bsigma_2,\dots,\bsigma_n)\;: \bsigma_i\in\reals^k, \|\bsigma_i\|_2=1\}$.
This can be identified with the product of $n$ $(k-1)$-dimensional spheres 
\begin{align}
\sS(n,k)\simeq \underbrace{\sS^{k-1}\times\sS^{k-1}\times\cdots \times
\sS^{k-1}}_{\mbox{$n$ times}}\, .
\end{align} 
The objective function $F_{\bA}:\sS(n,k)\to \reals$ is defined by:
\begin{align}
F_{\bA}(\bsigma) \equiv \sum_{i,j=1}^nA_{ij}\<\bsigma_i,\bsigma_j\>\, .
\end{align}
We then define 
\begin{align}
\OPT_k(\bA) \equiv \max\big\{ F_{\bA}(\bsigma)\; :\;\;\; \bsigma\in\sS(n,k)\, \big\}\, .
\label{eq:OptK}
\end{align}
Of course we have $\OPT_k(\bA)\le \SDP(\bA)$, with $\OPT_n(\bA)=\SDP(\bA)$.
The optimizer of the $k=n$ problem, $\bsigma_*$,  corresponds to an optimizer of the SDP through
$\bX_*=\bsigma_*\bsigma_*^{\sT}$. However, already much smaller values of $k$ give excellent (or exact) approximations
of the underlying SDP.
For instance, \cite{OursReplicas} used this approach to cluster graphs generated according to the 
sparse stochastic block model. Empirically, the resulting solutions are undistinguishable from the SDP optimum already
for $k=20$. The resulting algorithm can be used to cluster sparse graphs with up to $n=10^5$ vertices
in a matter of minutes\footnote{Code for the algorithm of  \cite{OursReplicas} is available at  
{\sf http://web.stanford.edu/~montanar/SDPgraph/}.}.

Unfortunately, the optimization problem (\ref{eq:OptK}) is non-convex and has --generically-- a large number of local 
maxima (but see Section \ref{sec:Related} for a discussion of related work). It is therefore unclear why  coordinate or gradient ascent should find a good approximation of $\OPT_k(\bA)$,
let alone a good approximation of $\SDP(\bA)$. Our main result shows that in fact all the maxima have value close
to the global optimum, and in fact close to $\OPT_k(\bA)$. (Here and below $\|\bM\|_2$ denotes the $\ell_2$
operator norm of matrix $\bM$.)
\begin{theorem}\label{thm:Main}
Let $\bsigma$ be a local maximum of $F_{\bA}$ over $\sS(n,k)$, $k\ge 2$. Then
\begin{align}
\SDP(\bA)\ge F_{\bA}(\bsigma)\ge \SDP(\bA) -\frac{8}{\sqrt{k}} \; n \|\bA\|_2\, .\label{eq:Main}
\end{align}
\end{theorem}
The proof of this theorem is provided in Section \ref{sec:Proofs}.

\begin{remark}
In typical applications $\|\bA\|_2 = \Theta(1)$, while $\SDP(\bA)=\Theta(n)$. In these cases
Theorem \ref{thm:Main} guarantees that an arbitrarily small relative error can be achieved by taking $k$ to be a 
large constant.

A simple example is the so-called $\integers_2$ synchronization problem \cite{OursReplicas}. In that case 
$\bA = (\lambda/n)\bxz\bxz^{\sT}+\bW$  with $\bxz\in\{+1,-1\}^n$, and $\bW$  a GOE radom matrix,
i.e. a symmetric matrix with independent entries $(W_{ii})_{1\le i\le n }$, $W_{ii}\sim \normal(0,2/n)$,
and $(W_{ij})_{1\le i<j\le n }$, $W_{ij}\sim \normal(0,1/n)$.
Basic random matrix theory implies $\|\bA\|_2\le 2+\lambda$ with high probability, while using $\bX = \bxz\bxz^{\sT}$
yields $\SDP(\bA)\ge n\lambda+o(n)$ (and indeed \cite{montanari2016semidefinite} 
proves $\SDP(\bA)\ge n\max(2,\lambda)+o(n)$).
\end{remark}

\begin{remark}
A naive application of Theorem \ref{thm:Main} to the sparse stochastic block model \cite{OursReplicas}
shows that an arbitrarily small error is achieved for $k = C\sqrt{\log n/\log\log n}$, with $C$ a large constant.
In fact this can be improved further to $k=O(1)$, by considering a suitably modified graph.
\end{remark}

\subsection{Further related work}
\label{sec:Related}

Burer and Monteiro \cite{burer2003nonlinear} introduced the the idea of constraining the rank and 
solving the PSD constraint thus obtaining a smooth non-convex problem. They also proved that, 
taking $k\ge \sqrt{2n}$, and under suitable conditions on $\bA$, the resulting non-convex problem 
has no local maxima, except for the global one. Their result actually extend to more general SDPs than 
Eq.~(\ref{eq:SDP.DEF}). 
While interesting, this result does not clarify the empirical finding that  $k=20$ is sufficient  
for some problems with $n$ as large as $10^5$ \cite{OursReplicas}.

Journ{\'e}e, Bach, Absil and Sepulchre \cite{journee2010low} proved sufficient conditions under 
which a local optimum of the rank-constrained problem is in fact a global optimum of the SDP.
In particular they proved that this happens if the local optimum $\bsigma\in\sS(n,k)$ is rank-deficient
(i.e. has rank at most $k-1$). It is however unclear \emph{a priori} when this happens,
and even computing the \emph{exact} rank of $\bsigma$ is very difficult (since $\bsigma$ is
typically produced by an ascent method). In the numerical
experiments of \cite{OursReplicas}, the local optimum was typically full rank.

Bandeira, Boumal and Voroninski  recently considered the extreme case $k=2$, in the specific
example of the $\integers_2$ synchronization problem \cite{bandeira2016low}. They proved that, if the 
signal is strong enough,
then this approach can effectively recover the underlying signal. Namely, all local minima are correlated with the signal. 

The SDP (\ref{eq:SDP.DEF}) was recently studied in the context of the so-called
community detection problem, namely for recovering vertex labels under the sparse stochastic block model.
Among a large number of interesting contributions, the closes to the present work are the ones concerning the 
so-called detection threshold  \cite{guedon2015community,montanari2016semidefinite}.
In particular, \cite{montanari2016semidefinite} proves a Grothendieck-type inequality for $\OPT_k(\bA)$. 
In slightly simplified form, this implies
\begin{align}
\OPT_k(\bA)\ge \SDP(\bA)- \frac{C}{k}\, n\|\bA\|_2\, ,
\end{align}
with $C$ an absolute constant. This can be immediately compared with Theorem 
\ref{thm:Main}, which of course implies $\OPT_k(\bA)\ge \SDP(\bA)- (9/\sqrt{k})\, n\|\bA\|_2$, 
This is similar to the result of \cite{montanari2016semidefinite}, but has a weaker dependence on $k$,
thus raising the interesting question of whether the 
dependence on $k$ in Theorem \ref{thm:Main} can be improved.

Grothendieck inequalities have found a broad range of applications
in computer science, see  \cite{khot2012grothendieck} and references therein.
However, they are typically used to approximate a combinatorial problem by solving a semidefinite program.
Theorem \ref{thm:Main} instead points at a different direction, namely solving the SDP with arbitrarily high accuracy 
by solving a non-convex rank-constrained problem. Of course this in turn provides an approximation of the 
original combinatorial problem.

Finally, there has been growing interest in non-convex methods for solving high-dimensional 
statistical estimation problems. Examples include matrix completion \cite{keshavan2010matrix},
phase retrieval \cite{chen2015solving}, regression 
with missing entries \cite{loh2011high}, and many others. These papers provide rigorous guarantees
under the assumption that the noise in the data is `small enough.' Under such conditions, a very good initialization 
can be constructed, e.g. by a spectral method, and it is sufficient to prove that the optimization
problem is well behaved in a neighborhood of the optimum.

The mechanism studied here  is dramatically different. Not only we do not require any `strong signal' condition,
but there is actually no signal at all (in other words, no `signal plus noise' structure is assumed).
We do not establish a property of a neighborhood of the correct solution, but instead a global property
of the cost function.

\subsection{Notations}

Throughout the paper, vectors  and matrices are denoted by boldface, e.g. $\bx,\by, \bz,\dots$.
We reserve the upper case, e.g. $\bA,\bB,\dots$, for $n\times n$ matrices.
Their coordinates are  denoted by non-bold font, e.g. $\bx = (x_1,\dots,x_n)$. We identify an element
$\bsigma = (\bsigma_1,\dots,\bsigma_n)\in\sS(n,k)$ with a matrix in $\reals^{n\times k}$.
The standard scalar product of $\bx,\by\in\reals^n$ is denoted by $\<\bx,\by\>=\sum_{i=1}^nx_iy_i$. 
Analogously, $\<\bA,\bB\> = \Tr(\bA\bB^{\sT})$ is the scalar product between matrices.

The eigenvalues of a symmetric matrix $\bM$ are denoted by $\xi_1(\bM)\ge\xi_2(\bM)\ge \cdots\ge \xi_n(\bM)=
\xi_{\min}(\bM)$. We let $\|\bx\|_2 = \sqrt{\<\bx,\bx\>}$ be the $\ell_2$ norm of vector $\bx$,
$\|\bM\|_2$ be the $\ell_2$ operator norm of matrix $\bM$, and $\|\bM\|_F$ its Frobenius norm.

\section{Proofs}
\label{sec:Proofs}

\subsection{Preliminary lemmas}

Our first lemma collects the classical second order conditions for problem (\ref{eq:OptK}).
While this is standard, we provide a self-contained proof for the reader's convenience.
\begin{lemma}
Let $\bsigma\in\sS(n,k)$ be a local maximum of $F_{\bA}$ on $\sS(n,k)$. Then there exists 
a diagonal matrix $\bLambda$ of Lagrange multipliers such that the following conditions hold:
\begin{enumerate}
\item \emph{First order stationarity:}
\begin{align}
(\bLambda-\bA)\bsigma = 0\,. \label{eq:FirstOrder}
\end{align}
\item \emph{Second-order stationarity:} for all $\bu = (\bu_1,\dots,\bu_n)$, $\bu_i\in\reals^k$, 
such that $\<\bsigma_i,\bu_i\>=0$, we have
\begin{align}
\sum_{i,j=1}^n(\bLambda-\bA)_{ij}\<\bu_i,\bu_j\>\ge 0\, .\label{eq:SecondOrder}
\end{align}
\end{enumerate}
\end{lemma}
\begin{proof}
Fix $\bu$ as in point 2 above an let $\bhu_i = \bu_i/\|\bu_i\|_2$ (with the convention $\bhu_i=0$ 
if $\bu_i = 0$). Define $\bsigma(t)$ for $t\in \reals$, by
\begin{align}
\bsigma_i(t) \equiv \bsigma_i\, \cos\big(\|\bu_i\|_2\, t\big)+ \bhu_i\, \sin\big(\|\bu_i\|_2\, t\big)\, .
\end{align}
Note that $\bsigma(t)\in \sS(n,k)$ for all $t\in\reals$, and $\bsigma(0)=\bsigma$. Also $t\mapsto \bsigma(t)$ 
is smooth.
Since $\bsigma$ is a local maximum of $F_{\bA}$, it follows that $t=0$ must be a local maximum of
$t\mapsto F_{\bA}(\bsigma(t))$. 

By Taylor expansion
\begin{align}
\bsigma_i(t)  = \bsigma_i + \bu_i\, t -\frac{1}{2} \|\bu_i\|_2^2\bsigma_i \, t^2+O(t^3)\, ,
\end{align}
and
\begin{align}
F_{\bA}(\bsigma(t)) &= F_{\bA}(\bsigma(0))+ 2\sum_{i=1}^n \<\bu_i,\bg_i\>\, t-
\sum_{i=1}^n \|\bu_i\|_2^2\<\bsigma_i,\bg_i\>\, t^2 +
\sum_{i,j=1}^nA_{ij}\<\bu_i,\bu_j\>\, t^2+ O(t^3)\, \label{eq:Taylor}\\
\bg_i &\equiv \sum_{j=1}^nA_{ij}\bsigma_j\, .
\end{align}
By first order stationarity of $t\mapsto F_{\bA}(\bsigma(t))$, we must have $\<\bu_i,\bg_i\>=0$
for all $\bu_i$ orthogonal to $\bsigma_i$. Therefore $\bg_i = \lambda_i\, \bsigma_i$ for some 
$\lambda_i\in\reals$. This yields Eq.~(\ref{eq:FirstOrder}) with $\bLambda$ the diagonal matrix with 
$\Lambda_{i,i}=\lambda_i$. 

Substituting in Eq.~(\ref{eq:Taylor}), we get
\begin{align}
F_{\bA}(\bsigma(t))&= F_{\bA}(\bsigma(0)) +
\sum_{i,j=1}^n(\bA-\bLambda)_{ij}\<\bu_i,\bu_j\>\, t^2+ O(t^3)\, ,
\end{align}
which immediately implies the claim (\ref{eq:SecondOrder}), since $t\mapsto F_{\bA}(\bsigma(t))$ is a local maximum.
\end{proof}

Throughout, for  $\Lambda$ the matrix of Lagrange multipliers, we 
will denote by $\lambda_i=\Lambda_{i,i}$ its diagonal entries. The next lemma collects a few simple facts
about local maxima.
\begin{lemma}\label{lemma:Simple}
Let $\bsigma\in\sS(n,k)$ be a local maximum of $F_{\bA}$ on $\sS(n,k)$.
Then we have the following
\begin{enumerate}
\item The associated matrix of multipliers $\bLambda$ is uniquely determined by 
\begin{align}
\lambda_i = \Big\|\sum_{j=1}^nA_{ij}\bsigma_j\Big\|_2\, .
\end{align}
\item $F_{\bA}(\bsigma) =\Tr(\bLambda)$.
\item For every $S\subseteq [n]$, $|S|\le k-1$, we have $(\bLambda-\bA)_{S,S}\succeq 0$.
In particular, $\lambda_i\ge 0$.
\item $\|\bLambda\|_F^2 = \sum_{i=1}^n\lambda_i^2\le n\|\bA\|_2^2$.
\item We have
\begin{align}
\xi_{\min}(\bsigma^{\sT}\bsigma)\le \frac{n}{k}\, .
\end{align}
\end{enumerate}
\end{lemma}
\begin{proof}
For point 1, note that, by Eq.~(\ref{eq:FirstOrder}), $\lambda_i \bsigma_i = \sum_{j=1}^nA_{ij}\bsigma_j$.
The claim follows by taking the norm of both sides since $\|\bsigma_i\|=1$ (and noting that $\lambda_i\ge 0$
as proved in point 3).

Point 2 follows again from $\lambda_i \bsigma_i = \sum_{j=1}^nA_{ij}\bsigma_j$, multiplying both sides by $\bsigma_i$
and summing over $i\in\{1,\dots, n\}$. 

For point $3$ assume, without loss of generality, that $S= \{1,2,\dots,k-1\}$. 
Let $\bu\in\reals^k$, $\|\bu\|_2=1$ be orthogonal to $\bsigma_1,\dots,\bsigma_{k-1}$. For any $\bx\in \reals^{k-1}$, define
$\bu_i = x_i\, \bu$ for $i\le k-1$, and $\bu_i =0$ for $i\ge k$. Then, the second order stationarity condition 
 (\ref{eq:SecondOrder}) implies 
\begin{align}
0\le \sum_{i,j=1}^n(\bLambda-\bA)_{ij}\<\bu_i,\bu_j\> = \<\bx,(\bLambda-\bA)_{SS}\bx\>\, ,
\end{align}
which proves the claim.

For point 4, note that, since $\bLambda\bsigma =\bA\bsigma$ we also have
\begin{align}
\sum_{i=1}^n\lambda_i^2 &= \<\bsigma,\bLambda^2\bsigma\> =  \<\bLambda\bsigma,\bLambda\bsigma\> \\
& = \<\bA\bsigma,\bA\bsigma\> = \Tr(\bA^2\bsigma\bsigma^{\sT}) \\
& \le \|\bA^2\|_2 \|\bsigma\bsigma^{\sT}\|_*\, ,
\end{align}
where $\|\bsigma\bsigma^{\sT}\|_*$ denotes the nuclear norm (sum of absolute values of singular values)
of matrix $\bsigma\bsigma^{\sT}$. The claim follows since 
\begin{align}
\|\bsigma\bsigma^{\sT}\|_* = \|\bsigma\|_F^2 =\sum_{i=1}^n\|\bsigma_i\|_2^2 = n\, .
\end{align}
(The first equality holds because the singular values of  $\bsigma\bsigma^{\sT}$ are obtained by squaring the singular values of $\bsigma$.)

For the last point let
\begin{align}
\bSigma \equiv \frac{1}{n}\bsigma^{\sT}\bsigma=\frac{1}{n}\sum_{i=1}^n\bsigma_i\bsigma^{\sT}_i \, .
\end{align}
Then $\bSigma\succeq 0$ and $\Tr(\bSigma) = \sum_{i=1}^n\|\bsigma_i\|_2^2/n = 1$. Hence
$\sum_{i=1}^k\xi_i(\bSigma) = 1$, and since $\xi_i(\bSigma)\ge 0$, we necessarily have $\xi_{\min}(\bSigma)\le 1/k$.
\end{proof}

\subsection{Proof of Theorem \ref{thm:Main}}

Let $\bsigma\in\sS(n,k)$ be a local maximum of $F_{\bA}$ on $\sS(n,k)$, and $\bLambda$
be the associated multipliers. By Lemma \ref{lemma:Simple} (point 5), we can fix $\bx\in\reals^k$
$\|\bx\|_2=1$ such that
\begin{align}
\bDelta \equiv \bsigma\bx\in\reals^{n}\, ,\;\;\;\;\;\; \|\bDelta\|_2 \le \sqrt{\frac{n}{k}}\, .\label{eq:DeltaBound}
\end{align}
Let $S\subseteq [n]$, be the set of indices $S\equiv\{i\in [n]:\; |\Delta_i|\ge 1/\sqrt{2}\}$. 
By Markov inequality $|S|\le 2n/k$. 

For a vector $\bz\in\reals^n$, let $\tbz$ be the vector obtained by zero-ing the entries in $S$.
Define $\bu = (\bu_1,\dots, \bu_n)$,  $\bu_i\in\reals^k$, by
\begin{align}
\bu_i = \tz_i\,\bx +\tbu_i\;, \;\;\;\; \tbu_i \equiv -\frac{\tz_i\Delta_i}{1-\Delta_i^2}\, \projp_{\bx}\bsigma_i\, ,
\end{align}
where $\projp_{\bx} = \id-\bx\bx^{\sT}$ is the projector orthogonal to $\bx$.
Here, it is understood that $\tbu_i = 0$ when $|\Delta_i|=1$.

Note that $\Delta_i = \<\bsigma_i,\bx\>\in [-1,1]$. Hence $\|\projp_{\bx}\bsigma_i\|_2^2 = \<\bsigma_i,\projp_{\bx}\bsigma_i\>= 1-\Delta_i^2$. We therefore have
\begin{align}
\<\bsigma_i,\bu_i\> = \tz_i\, \Delta_i - \frac{\tz_i\Delta_i}{1-\Delta_i^2}\<\bsigma_i,\projp_{\bx}\bsigma_i\>=0\, .
\end{align}
We can therefore apply the second-order stationarity condition (\ref{eq:SecondOrder}), to get
(with the notation $\bv_R = (\bv_i:\; i\in R)$, for $R\subseteq[n]$):
\begin{align}
0&\le \sum_{i,j=1}^n(\bLambda-\bA)_{i,j}\<\bu_i,\bu_j\>\\
& = \sum_{i,j=1}^n(\bLambda-\bA)_{i,j} \tz_i\tz_j+\sum_{i,j=1}^n(\bLambda-\bA)_{i,j} \<\tbu_i,\tbu_j\>\\
&\le \<\tbz,(\bLambda-\bA)\tbz\>+\sum_{i=1}^n \lambda_i\|\tbu_i\|_2^2+
\|\bA\|_2\|\tbu\tbu\|_*\\
&\le \<\tbz,(\bLambda-\bA)\tbz\>+\sum_{i=1}^n \lambda_i\|\tbu_i\|_2^2+
\|\bA\|_2\|\tbu\|_F^2\, ,\label{eq:Continuing}
\end{align}
where in the last inequality $\|\bM\|_*$ denotes the nuclear norm of matrix $\bM$ (sum of absolute values of singular 
values), and we used the fact that $\|\bM\bM^{\sT}\|_*= \|\bM\|_F^2$.
Now 
\begin{align}
\|\tbu\|_F^2 &= \sum_{i=1}^n\|\tbu_i\|_2^2 \\
& =\sum_{i=1}^n\frac{\tz_i^2\Delta_i^2}{1-\Delta_i^2}\\
& \le 2 \sum_{i=1}^n\tz_i^2\Delta_i^2\, .
\end{align}
Therefore, continuing from Eq.~(\ref{eq:Continuing}), ad using a similar bound for the first sum,
we get the inequality
\begin{align}
0\le \<\bz_{S^c},(\bLambda-\bA)_{S^c,S^c}\bz_{S^c}\>+2\sum_{i\in S^c} \lambda_iz_i^2\Delta_i^2 +2 
\|\bA\|_2\sum_{i\in S^c} z_i^2\Delta_i^2\, ,
\end{align}
where we made explicit the dependence on $S$.

Now consider $\bv\in\sS(n,n)$, i.e. $\bv = (\bv_1,\bv_2,\dots,\bv_n)$, $\bv_i\in\sS^{n-1}$.
Applying the last inequality to the $a$-th column of $\bv$, denoted by $\bv^a$, we get
\begin{align}
0\le \<\bv^a_{S^c},(\bLambda-\bA)_{S^c,S^c}\bv^a_{S^c}\>+2\sum_{i\in S^c} \lambda_i(v^a_i)^2\Delta_i^2 +2 
\|\bA\|_2\sum_{i\in S^c} (v_i^a)^2\Delta_i^2\, .
\end{align}
Summing over $a\in\{1,\dots,n\}$ and using $\sum_{a=1}^n(v^a_i)^2=1$, we get
\begin{align}
0\le \<\bv_{S^c},(\bLambda-\bA)_{S^c,S^c}\bv_{S^c}\>+2\sum_{i\in S^c} \lambda_i\Delta_i^2 +2 
\|\bA\|_2\sum_{i\in S^c} \Delta_i^2\, .
\end{align}
Using Lemma \ref{lemma:Simple}, points 2 and 3 (in particular, $\lambda_i\ge 0$), we get
\begin{align}
\<\bv_{S^c},\bA_{S^c,S^c}\bv_{S^c}\> &\le F_{\bA}(\bsigma) + 2\sum_{i\in S^c} \lambda_i\Delta_i^2 +2 
\|\bA\|_2\sum_{i\in S^c} \Delta_i^2\\
&\le F_{\bA}(\bsigma) + \sqrt{2}\sum_{i\in S^c} \lambda_i|\Delta_i| +2 
\|\bA\|_2\|\bDelta\|_2^2\\
&\le F_{\bA}(\bsigma) + \sqrt{2}\|\bLambda\|_F\|\bDelta\|_2 +2 
\|\bA\|_2\|\bDelta\|_2^2\, ,
\end{align}
where the second inequality follows since $|\Delta_i|\le 1/\sqrt{2}$ for $i\in S$, and the last one by Cauchy-Schwartz.
Finally using Lemma \ref{lemma:Simple}, point 4, together with  Eq.~(\ref{eq:DeltaBound}), we get
\begin{align}
\<\bv_{S^c},\bA_{S^c,S^c}\bv_{S^c}\> &\le F_{\bA}(\bsigma) +\sqrt{\frac{8}{k}}\; n\|\bA\|_2\, .
\label{eq:Difficult}
\end{align}
Next notice that 
\begin{align}
\<\bv_S,\bA\bv_{S^c}\> &\le \|\bA_{S,S^c}\|_2\|\bv_S\bv_{S^c}^{\sT}\|_*\\
& \le  \|\bA\|_2\|\bv_S\|_F\|\bv_S^c\|_F = \|\bA\|_2\sqrt{|S|(n-|S|)}\\
& \le \sqrt{\frac{2}{k}}\; n\|\bA\|_2\, . \label{eq:Easy1}
\end{align}
Proceeding analogously, we get
\begin{align}
\<\bv_S,\bA\bv_{S}\> &\le \frac{2}{k}\; n\|\bA\|_2\, . \label{eq:Easy2}
\end{align}
Finally, combining the bounds (\ref{eq:Difficult}), (\ref{eq:Easy1}), (\ref{eq:Easy2})we get
\begin{align}
\<\bv,\bA\bv\> &\le F_{\bA}(\bsigma) +\frac{5\sqrt{2}}{\sqrt{k}}\; n\|\bA\|_2\, .
\end{align}
Since this holds for any $\bv\in\sS(n,n)$, it implies the second inequality in Eq.~(\ref{eq:Main}),
because $5\sqrt{2}<8$. The first inequality in Eq.~(\ref{eq:Main}) is trivial.

\subsection*{Acknowledgments}

This work was partially supported by NSF grants CCF-1319979 and DMS-1106627 and the
AFOSR grant FA9550-13-1-0036.

\bibliographystyle{amsalpha}

\providecommand{\bysame}{\leavevmode\hbox to3em{\hrulefill}\thinspace}
\providecommand{\MR}{\relax\ifhmode\unskip\space\fi MR }
\providecommand{\MRhref}[2]{%
  \href{http://www.ams.org/mathscinet-getitem?mr=#1}{#2}
}
\providecommand{\href}[2]{#2}

\end{document}